\documentclass{amsart}

\usepackage{amsfonts}
\usepackage{amsthm}
\usepackage{amssymb}
\usepackage{amsmath}
\usepackage[T1]{fontenc}
\usepackage[utf8]{inputenc}
\usepackage{mathrsfs}
\usepackage[retainorgcmds]{IEEEtrantools}
\usepackage[all]{xy}
\usepackage{verbatim}
\usepackage{tikz}

\theoremstyle{plain} 
\newtheorem{Lemma}{Lemma}[section] \newtheorem{Thm}[Lemma]{Theorem} \newtheorem{Prop}[Lemma]{Proposition}\newtheorem{Cor}[Lemma]{Corollary}  
\theoremstyle{definition} \newtheorem{Defn}[Lemma]{Definition} \newtheorem{Ex}[Lemma]{Example} \newtheorem{Alg}[Lemma]{Algorithm}
\theoremstyle{remark} \newtheorem{Rem}[Lemma]{Remark}  
\numberwithin{equation}{section}

\newcommand{\CC}{\mathbb{C}}
\newcommand{\RR}{\mathbb{R}}
\newcommand{\OO}{\mathcal{O}}

\newcommand{\dd}{\partial}

\newcommand{\inj}{\hookrightarrow}
\newcommand{\surj}{\twoheadrightarrow}

\DeclareMathOperator{\cok}{cok}
\DeclareMathOperator{\codim}{codim}
\DeclareMathOperator{\vol}{vol}
\DeclareMathOperator{\supp}{supp}
\newcommand{\sse}{\subset}
\newcommand{\ten}{\otimes}

\begin{document}

\author{Ketil Tveiten}
\title{B-Splines, Polytopes and Their Characteristic D-Modules}
\address{Ketil Tveiten\\Matematiska Institutionen\\ Stockholms Universitet\\ 106 91 Stockholm.}
\email{ktveiten@math.su.se}
\subjclass[2000]{14F10, 32C38, 52B11}
\keywords{$D$-modules, $B$-splines}
\begin{abstract}
  Given a polytope $\sigma\sse\RR^m$, its characteristic distribution $\delta_\sigma$ generates a $D$-module which we call the \emph{characteristic $D$-module of $\sigma$} and denote by $M_\sigma$. More generally, the characteristic distributions of a cell complex $K$ with polyhedral cells generate a $D$-module $M_K$, which we call the characteristic $D$-module of the cell complex. We prove various basic properties of $M_K$, and show that under mild topological conditions on $K$, the $D$-module theoretic direct image of $M_K$ coincides with the module generated by the $B$-splines associated to the cells of $K$ (considered as distributions). We also give techniques for computing $D$-annihilator ideals of polytopes.
\end{abstract}
\maketitle

\section{Introduction}

This paper concerns the $D$-module generated by the characteristic function of a polytope (or polyhedral cell complex) in $\RR^m$; we will call it the \emph{characteristic $D$-module} of the polytope. To consider the $D$-module generated by a function is a very natural construction, and in contrast to what is generally the case, the geometric content is very explicit here, encoded via the submodules generated by the characteristic functions of the faces of the polytope.

De Concini and Procesi in \cite{DeConcini-Procesi} give an extensive treatment of certain $B$-splines (in particular, those arising from the projection of a box or a coordinate orthant, these being most useful for applications), using tools from $D$-module theory and combinatorics. A $B$-spline is a function given by integrating over the fibers of a projection, which is precisely the kind of construction the $D$-module theoretic direct image is intended to capture (see e.g. \cite{Pham}). 
We can observe that the $D$-module generated by a $B$-spline should correspond in this sense to the direct image of the characteristic $D$-module of a suitable polytope. The aim of this paper is primarily to make this correspondence explicit, and give criteria for when it holds precisely; and so to provide a description of a class of $D$-module direct images, examples of which are in short supply in the literature. 

Sections \ref{sec:prelims} and \ref{sec:ann-id} describe the characteristic modules; these are constructed from linear semialgebraic sets in a very natural way. Section \ref{sec:ann-id} in particular answers another interesting question: what differential equations do the characteristic functions satisfy, equivalently what is their $D$-annihilator ideal? This seems to have been an open problem for arbitrary polytopes, and a complete (though not efficient) solution is given here. Section \ref{sec:dir-im} describes the direct images of the characteristic modules, including their higher direct images, and (Theorem \ref{Thm:mainThm}) their connection to the $B$-spline module of De Concini and Procesi.

We will fix the following notation: $X,Y$ and $Z$ denote $\CC^m,\CC^s$ and $\CC^{m-s}$, respectively. For any affine space $H$, $\CC[H]$ denotes the ring of polynomials on $H$, $D_H$ the ring of polynomial-coefficient differential operators on $H$; $D_H$ is called the \emph{Weyl algebra} (in $\dim(H)$ variables), and is isomorphic to the $\CC$-algebra generated by variables $x_1,\ldots,x_{\dim(H)},\dd_1,\ldots,\dd_{\dim(H)}$ subject to the relations $[\dd_i,x_i]=1$ and all other elements commute. We give $\RR^m$ its standard Euclidean structure, with induced Lebesgue measure $dx$ on $\RR^m$ and all its subspaces; and similarly for $X$. We denote the standard basis of $\RR^m$ by $e_1,\ldots,e_m$.


\section{Polyhedral cell complexes and the characteristic module}\label{sec:prelims}

\begin{Defn}
We will by \emph{polytope} or \emph{polyhedral cell} mean a closed simply-connected semialgebraic set in $\RR^m$ defined by linear polynomials, with dimension equal to the dimension of its affine hull. We do not make any requirements on convexity or compactness, but the requirement of closedness is essential. A \emph{face} of a polytope $\sigma$ is a polytope contained in $\sigma$, defined by the same polynomials as $\sigma$, with some inequalities replaced by equalities; in particular $\sigma$ is a face of itself, each component of the boundary of $\sigma$ is a face of $\sigma$, and the empty set is a face of every polytope. 
A \emph{facet} is a face of codimension one. A \emph{vertex} is a face of dimension zero.

By \emph{polyhedral cell complex} we mean a union of polyhedral cells subject to the requirement that the intersection of any two cells is a face of both (this is like a simplicial complex, but we allow more general cells). In particular, the cell complex consisting of a single polytope $\sigma$ and all its faces, is denoted $\widehat{\sigma}$.

The affine hull of a polytope $\sigma$ is denoted by $H_\sigma$.
\end{Defn}

\begin{Defn}
  The \emph{characteristic distribution} $\delta_\sigma$ of a polytope $\sigma$ is defined by
$$\delta_\sigma(\phi)= \int_\sigma \phi\; dx$$
for a test function $\phi$, where $\int_\sigma$ is the $\dim(\sigma)$-dimensional integral taken with respect to the appropriate restriction to $H_\sigma$ of the standard measure. The Weyl algebra acts on $\delta_\sigma$ by 
$$(p(x)\dd^\alpha\cdot\delta_\sigma)(\phi)=\int_\sigma (-1)^{|\alpha|}\dd^\alpha(p(x)\phi(x))\;dx$$
\end{Defn}

If we denote the facets of a cell $\sigma$ by $\sigma_i$, $i = 1,\ldots, r$, and let their outward unit normal vectors (relative to $H_\sigma$) be denoted $n_i$, we have the following relations, which we call the \emph{standard relations}.

\begin{Prop}[Standard relations]\label{Prop:relations}
\begin{itemize}
\item[(i)]If $\dim(\sigma)>0$, then for any directional derivative $\dd_v$ where $v$ is a vector tangent to $H_\sigma$, we have
$$\dd_v\cdot\delta_\sigma=-\sum_i\langle v|n_i\rangle\delta_{\sigma_i}.$$
\item[(ii)]Let $I(\sigma)$ denote the defining ideal of $H_\sigma$. For any $p\in I(\sigma)$, we have $$p\cdot\delta_\sigma=0.$$ As $\sigma$ is a polyhedral body, $H_\sigma$ is an affine space defined by $m-dim(\sigma)$ equations of degree 1, and the corresponding polynomials generate $I(\sigma)$.
\end{itemize}
\end{Prop}

\begin{proof}
  $(i)$ is Stokes' theorem, and $(ii)$ is clear.
\end{proof}

The set of standard relations for a polytope $\sigma\sse\RR^m$ is spanned by $\dim(\sigma)$ relations of type \ref{Prop:relations}$(ii)$ and $\codim(\sigma)$ relations of type \ref{Prop:relations}$(i)$; this is because $H_\sigma$ is defined by $\codim(\sigma)$ equations.

\begin{Ex}\label{Ex:unit-interval}
  Let $I$ be the unit interval $[0,1]$ in $\RR^1$, with coordinate $x$. Then $\dd_x\cdot\delta_I = \delta_0 - \delta_1$, $x\cdot\delta_0=0$, and $(x-1)\cdot\delta_1=0$. It follows that $\delta_I$ is annihilated by the operator $x(x-1)\dd_x$.
\end{Ex}

\begin{Defn}
  For a polyhedral cell complex $K=\bigcup \sigma$, we define the \emph{characteristic module} of $K$ to be the $D_X$-module
$$M_K:=D_X\cdot\{\delta_\sigma|\sigma\subset K\},$$
 generated by the characteristic distributions of all the cells of $K$.
\end{Defn}

\begin{Rem}\label{Rem:support}
  We note that the support (in $\CC^m$) of a generator $\delta_\sigma$ of $M_K$ is equal to the affine closure $H_\sigma$. This is different from the support of $\delta_\sigma$ considered as a distribution (on $\RR^m$), and is due to the fact that the module is defined by the differential equations the distribution $\delta_\sigma$ satisfies, which do not uniquely determine $\delta_\sigma$; there are other distributional solutions, but they are all supported on $H_\sigma$. In the remainder, the \emph{support of $\delta_\sigma$} will always mean the support of the generator $\delta_\sigma$.
\end{Rem}


\subsection{The skeleton filtration}\label{subsec:skeleton-filtration}

This module has a natural filtration by the dimension of the support of the generators:

\begin{Defn}[The skeleton filtration]\label{Defn:skel-filt} Let $F^iM_K:=D_X\cdot\{\delta_\sigma|\sigma\subset K, \dim(\sigma)\le i\}$, this is a submodule of $M_K$. These submodules form a filtration $$F^0M_K\sse F^1M_K \sse \cdots \sse F^{m-1}M_K\sse F^mM_K=M_K$$ which we call the \emph{skeleton filtration}.
\end{Defn}

\begin{Prop}\label{Prop:filtr-quots}
Let $i:H_\sigma\inj X$ be the inclusion map. The filtration quotients $Q_k:=F^kM_K/F^{k-1}M_K$ are semisimple, with summands isomorphic to the direct image under the inclusion $i_+^0\CC[H_\sigma]$, one for each $k$-cell $\sigma\sse K$. 
\end{Prop}

\begin{Rem}See \cite{Malgrange},\cite[V]{Borel} for definitions of the direct image functor $i_+$; this is a functor between the corresponding derived categories, and we will denote by $i_+^0$ the restriction to the zeroth cohomology object. When $i$ is a closed embedding as here, these are equivalent, and we also have the celebrated theorem of Kashiwara (\cite{Kashiwara-thesis}), which we will use several times in the rest of the paper.
  \begin{Thm}\label{Thm:Kashiwara}
    Let $i:V\inj W$ be a closed immersion of schemes. Then the functor $i_+$ is an equivalence between the category of coherent $D_V$-modules and the category of coherent $D_W$-modules with support on $V$.
  \end{Thm}
Kashiwara's theorem is very useful, among other things it allows us to assume we are in maximal dimension when we need to. See \cite[IV]{Malgrange} or \cite{Coutinho} for elegant expositions of the proof. 
\end{Rem}

\begin{proof}[Proof of \ref{Prop:filtr-quots}]
It is clear that $Q_k$ is generated by the (classes of the) $k$-cells, namely $Q_k=\sum D_X\cdot\overline{\delta_\sigma}$. We must show two things: that $D_X\cdot\overline{\delta_\sigma}$ is of the given form, and that the sum is direct.

We may assume by choosing coordinates appropriately that $H_\sigma$ is the affine flat $x_{k+1}-p_{k+1}=\cdots = x_m-p_m = 0$. From the standard relations given in Proposition \ref{Prop:relations}, it follows that 
$$\begin{array}{rl} 
\dd_j\overline{\delta_\sigma}=0, & j\le k \\ 
(x_j-p_j)\overline{\delta_\sigma}=0, & j>k.
\end{array}$$
Indeed, we have $\dd_v\cdot\delta_\sigma=\sum_j\langle v|n_j\rangle\delta_{\sigma_j}$ for $v$ parallel to $H_\sigma$, and in the quotient the right-hand side disappears, so we are left with $\dd_v\overline{\delta_\sigma}=0$.

The existence of these relations implies that there is a surjective map
$$i_+^0\CC[H_\sigma]\to D_X\cdot\overline{\delta_\sigma}$$
and as the first module is simple by Kashiwara's Theorem (\ref{Thm:Kashiwara}), this is an isomorphism unless $D_X\cdot\overline{\delta_\sigma}$ is the zero module. It is not, as $\overline{\delta_\sigma}=0$ would imply that $\delta_\sigma$ is some linear combination of distributions with support on lower-dimensional cells, which cannot be true as their supports have different dimension. Directness of the sum follows easily.
\end{proof}

\begin{Cor}\label{Cor:simple}
  The modules $D_X\cdot\delta_{H_\sigma}$ and $D_X\cdot\overline{\delta_\sigma}$ are isomorphic, and also simple.
\end{Cor}

\begin{Rem}\label{Rem:tool}
  This implies that we can write $D_X\cdot\overline{\delta_\sigma}$ as $\CC[x_1,\ldots,x_k,\dd_{k+1},\ldots,\dd_m]$ for suitable coordinates $x_i$, where $k=\dim(\sigma)$. Suppose the coordinates are chosen so that $H_\sigma=\{x_{k+1}=p_{k+1},\ldots,x_m=p_m\}$, then for $i\le k$, $\dd_i$ acts by $\dd_i\cdot x_i = 1$, and for $i>k$, $x_i$ acts by $x_i\cdot\dd_i=p_i\dd_i-1$.
\end{Rem}

Because we now have a composition series for $M_K$ with regular holonomic quotients, and regular holonomicity is preserved under extensions, we deduce the following:

\begin{Cor}\label{Cor:holonomic}
  $M_K$ is regular holonomic.
\end{Cor}

\begin{Prop}\label{Prop:presentation}
$M_K$ is the quotient of the free module generated by the cells of $K$, by the submodule generated by the standard relations given in Proposition \ref{Prop:relations}. Letting $c$ be the number of cells in $K$, this submodule is generated by a total of $m\cdot c$ 
relations, and hence there is a canonical presentation $$D_X^{m\cdot c}\to D_X^{c}\surj M_K,$$ where the last map is given by $\sum_{\sigma\sse K}p_\sigma\cdot g_\sigma\mapsto \sum_{\sigma\sse K} p_\sigma\delta_{\sigma}$, and $D_X^c=\bigoplus D_X\cdot g_\sigma$ is the free module generated by the cells of $K$.
\end{Prop}

\begin{proof}
Let us first define the maps properly. We label the generators of $D_X^c$ by the cells of $K$ --- so that $D_X^c$ is freely generated by generators $g_\sigma$, for all the $\sigma\sse K$ --- and let the map $D_X^c\to M_K$ be given by $g_\sigma\mapsto \delta_\sigma$. 

For each cell $\sigma\sse K$, the standard relations of type $(i)$ and $(ii)$ form vector spaces of dimension $\dim(H_\sigma)$ and $m-\dim(\sigma)$ respectively, so for each $\sigma$ there are $m$ linearly independent (over $\CC$) relations that generate all. Each can be written as a $D_X$-linear combination $P^\sigma(\delta_\sigma,\ldots,\delta_{\sigma_k})=0$.

We now let $D_X^r$ be freely generated by generators $r_{P^\sigma}$, one for each generating standard relation, and define the map $D_X^r\to D_X^c$ by $r_{P^\sigma}\mapsto P^\sigma(g_\sigma, \ldots,g_{\sigma_k})$.

The skeleton filtration on $M_K$ induces filtrations on $D_X^c$ and $D_X^r$, in both cases by dimension of $\sigma$: $F'^iD_X^c$ and $F''^iD_X^r$ are generated, respectively, by those $g_\sigma$ and $r_{P^\sigma}$ with $\dim(\sigma)\le i$. Both maps respect the filtration, so passing to the associated graded modules we see that $gr(D_X^r)\to gr(D_X^c)\surj gr(M_K)$ is a direct sum of sequences $D_X^m\stackrel{\alpha_\sigma}{\to}D_X\to D_X\cdot\overline{\delta_\sigma}$, one for each cell $\sigma$, where $\alpha_\sigma$ is the map given by given by $r_{P^\sigma}\mapsto P^\sigma(g_\sigma,0,\ldots,0)$. Only exactness in the middle is non-obvious. Choosing coordinates such that $H_\sigma$ is given by $x_{k+1}=\cdots =x_m=0$, so that $D_X\cdot\overline{\delta_\sigma}\simeq \CC[x_1,\ldots,x_k,\dd_{k+1},\ldots,\dd_m]$ as in \ref{Rem:tool}, we see that the cokernel of $\alpha_\sigma$ is $D_X/(\sum_{i\le k} D_X\cdot\dd_i + \sum_{i>k}D_X\cdot x_i)$, and this is clearly isomorphic to $\CC[x_1,\ldots,x_k,\dd_{k+1},\ldots,\dd_m]\simeq D_X\cdot\overline{\delta_\sigma}$.
\end{proof}

Using the canonical presentation, it is easy to show the following facts.

\begin{Cor}\label{Cor:etc}
If $K\sse L$ is a subcomplex, closed in $L$, then $M_K$ is a submodule of $M_L$. 
If $K_1$ and $K_2$ are glued along a subcomplex $F$, we have $M_{K_1\cup_F K_2}\simeq M_{K_1}\oplus_{M_F}M_{K_2}$.
\end{Cor}

\begin{Thm}\label{Thm:cyclic} Recall that $\widehat{\sigma}$ denotes the cell complex consisting of a polytope $\sigma$ and all its faces. Assume for all faces $\alpha,\beta$ of $\sigma$ that if $H_\alpha\sse H_\beta$, then $\alpha$ is a face of $\beta$. Then $M_{\widehat{\sigma}}\simeq D_X\cdot\delta_\sigma \simeq D_X/Ann_{D_X}(\delta_\sigma)$.
\end{Thm}

\begin{proof}
  The claim is true if for any $k$-face $\tau$ we can find a $P\in D_X$ such that $P\cdot\delta_\sigma=\delta_\tau$, and it suffices by repeated application to assume $\tau$ is a facet. Let now $A$ be the set of 1-faces $\alpha$ of $\sigma$ not lying in $H(\tau)$, that is, $\alpha$ is not a face of $\tau$. For each $\alpha\in A$, let $\dd_\alpha$ be a directional derivative along $\alpha$. As each face $\beta$ of $\sigma$ not in $H(\tau)$ is parallel to some $\alpha$, by \ref{Prop:relations} the action of $\dd_\alpha$ reduces $\delta_\beta$ to a sum of terms with support on the facets of $\beta$, specifically those facets not having $\alpha$ as a face. Consider now $(\prod_{\alpha\in A}\dd_\alpha)\cdot\delta_\sigma$, this will by the previous observation and the assumption on supports be equal to a sum of point distributions $\delta_p$ for points $p\not\in H(\tau)$, generators $\delta_{\gamma}$ for $\gamma$ a face of $\tau$, and $\delta_\tau$, each with coefficient some polynomial $P_p(\dd),P_\gamma(\dd)$ in the variables $\dd_i$. We make the claim that the degree of $P_\tau(\dd)=\prod_{\alpha\in A}\dd_\alpha$ is strictly larger than the degree of any other $P_p(\dd),P_\gamma(\dd)$. The reason is that because terms involving $\delta_p,\delta_\gamma$ are obtained by applying $\dd_\alpha$ to some $\delta_\beta$ with $\beta$ parallel to $\alpha$, and as we have the standard relation $\dd_\alpha\delta_\beta=\sum_jc_j\delta_{\beta_j}$ (where $\beta_j$ are the faces of $\beta$ not parallel to $\alpha$, and $c_j$ are constants), the polynomial coefficients of the $\delta_{\beta_j}$ have lower degree than the coefficient polynomial of $\delta_\beta$. As $\delta_\tau$ is by assumption not parallel to any $\alpha$, there are no standard relations reducing $\dd_\alpha\delta_\tau$ to a sum of $\delta_\gamma$'s, and so the coefficient of $\delta_\tau$ retains the maximal degree. 

So, $(\prod_{\alpha\in A}\dd_\alpha)\cdot\delta_\sigma = (\prod_{\alpha\in A}\dd_\alpha)\cdot\delta_\tau + \sum_pP_p(\dd)\delta_p+\sum_{\gamma\sse\tau}P_\gamma(\dd)\delta_\gamma$ with $\deg(P_p)$ and $\deg(P_\gamma)$ both strictly less than $\deg(\prod_{\alpha\in A}\dd_\alpha)=|A|$. Now, $H(\tau)$ is a hyperplane, and we may assume its defining equation is $x_m=0$. Each $p=(p_1,\ldots,p_m)$ occuring here lies in a hyperplane $x_m=p_m$, so acting on $(\prod_{\alpha\in A}\dd_\alpha)\cdot\delta_\tau + \sum_pP_p(\dd)\delta_p+\sum_{\gamma\sse\tau}P_\gamma(\dd)\delta_\gamma$ by $(x_m-p_m)^{|A|}$ kills $P_p(\dd)\delta_p$ (as $\deg(P_p)<|A|$), and moreover (as we are multiplying with a polynomial in the $x_i$ variables) clearly does not increase the degree of any other coefficient $P_*(\dd)$ in the sum (as these are polynomials in the $\dd_i$ variables). In this way, kill off the sum $\sum_pP_p(\dd)\delta_p$, and we are left with $C\cdot(\prod_{\alpha\in A}\dd_\alpha)\cdot\delta_\tau +\sum_{\gamma\sse\tau}P_\gamma(\dd)\delta_\gamma$ (where $C$ is some constant), and acting on this by $x_m^{|A|}$ we kill the sum $\sum_{\gamma\sse\tau}P_\gamma(\dd)\delta_\gamma$, and the term $C\cdot(\prod_{\alpha\in A}\dd_\alpha)\cdot\delta_\tau$ is reduced to some constant times $\delta_\tau$.
\end{proof}


\subsection{De Rham cohomology of $M_K$}\label{subsec:homology}

\begin{Defn}\label{Defn:deRham}
  The \emph{de Rham complex} $DR_X(M)$ of a left $D_X$-module $M$ is the complex $\Omega_X^\bullet\ten_{\CC[X]} M[m]$, with differential $d(\omega\ten m)=d\omega\ten m+\sum_idx_i\land\omega\ten\dd_im$.
\end{Defn}

\begin{Thm}\label{Thm:deRham}
  The de Rham complex $DR_X(M_K)$ of $M_K$ is quasi-isomorphic to the Borel-Moore homology chain complex $C_\bullet^{BM}(K,\CC)$.
\end{Thm}

\begin{proof}
  We observe that the differential in the de Rham complex respects the skeleton filtration, and because each filtration quotient $Q_k$ has support only of dimension $k$, the associated spectral sequence collapses on the $E_1$ page, with $E_1^{pq}=H_{dR}^{p+q}(Q_{-p})\simeq \CC^{a_k}$ if $(p,q)=(0,-k)$ and zero otherwise, to a single row 
$$\CC^{a_m}\to\CC^{a_{m-1}}\to\cdots\to\CC^{a_0},$$
where $a_k$ is the number of cells $\sigma \sse K$ with $\dim(\sigma)=k$ and $\CC^{a_k}$ is the vector space with generators $\omega_\sigma\ten\delta_\sigma:=dx_{k+1}\land\cdots\land dx_{m}\ten \overline{\delta_\sigma}$ (in coordinates such that $H_\sigma$ is parallel to $x_1=\cdots=x_k=0$) for each such $\sigma$.
It suffices to check a single generator (assuming suitable coordinates).
\begin{IEEEeqnarray*}{rCl} d(\omega_\sigma\ten\overline{\delta_\sigma})&=& \sum_{i=1}^k  dx_i\land \omega_\sigma\ten\dd_i\overline{\delta_\sigma} \\ 
&=& -\sum_{i=1}^k \sum_{j}\langle e_i|n_{j}\rangle dx_i\land \omega_\sigma\ten \overline{\delta_{\sigma_j}} \\ &=& -\sum_j\langle \sum_{i=1}^k e_i|n_j\rangle dx_i\land \omega_\sigma\ten\overline{\delta_{\sigma_j}}\\
&=& -\sum_jd(\sum_{i=1}^k e_i|n_j\rangle x_i)\land \omega_\sigma\ten\overline{\delta_{\sigma_j}}\\
&=& -\sum_j d(n_j)\land \omega_\sigma\ten\overline{\delta_{\sigma_j}}\\
&=& -\sum_j \omega_{\sigma_j}\ten\overline{\delta_{\sigma_j}}
\end{IEEEeqnarray*}
We see that the generator corresponding to $\delta_\sigma$ is sent to the sum of the generators corresponding to the boundary cells $\delta_{\sigma_j}$. Note that each generator $[\sigma]$ has closed support; $d$ thus corresponds to the boundary maps for chains of \emph{closed support}, i.e. the Borel-Moore homology boundary map, and we are done.

Quasi-isomorphism follows from the observation that the map $C_\bullet^{BM}\to DR_X(M_K)$ sending a homology class $[\sigma]$ to its corresponding generator $\omega_\sigma\ten\delta_\sigma$ is an injective chain map which by the above is the identity on (co)homology.
\end{proof}


\begin{Rem}\label{Rem:functor}
  The modules $M_K$ generate by taking extensions a subcategory of the category of regular holonomic $D_X$-modules, and it follows from the existence of the skeleton filtration that this category is equal to the category of regular holonomic $D_X$-modules admitting a composition series with quotients each isomorphic to $i_+^0\OO_{H}$ for some affine flat $H$. The assignment $K\mapsto M_K$ is a functor into this category from the category of polyhedral cell complexes and cellular maps, 
this functor is faithful and preserves finite limits and colimits. Moreover, \ref{Thm:deRham} gives us a commutative diagram of functors:
\begin{center}  \begin{tikzpicture}[node distance=2cm,auto]
    \node (K) {$K$};
    \node (M) [right of=K]{$M_K$};
    \node (C) [below right of=K]{$DR(M_K)$};
    \draw[|->] (K) to node {}(M);
    \draw[|->] (K) to node[swap]{{\tiny $C_\bullet^{BM}$}}(C);
    \draw[|->] (M) to node {}(C);
  \end{tikzpicture} \end{center}
\end{Rem}


\section{Annihilator ideals for $\delta_\sigma$}\label{sec:ann-id}

Recall that the polyhedral cell complex consisting of a single polytope $\sigma$ and all its faces, is denoted by $\widehat{\sigma}$. 
In this section, we apply our constructions to produce some tools that enable computation of the annihilator ideal $Ann_{D_X}(\delta_\sigma)$ for any polytope $\sigma$. By application of Kashiwara's theorem we may assume $\sigma$ is of maximal dimension. 

\begin{Prop}\label{Prop:ann-delta-sigma}
  For each vertex $p$ of $\sigma$, let $C_p$ be the 
cone at $p$ spanned by the faces intersecting $p$. 
Then $Ann_{D_X}(\delta_\sigma)=\cap_p Ann_{D_X}(\delta_{C_p})$.
\end{Prop}

\begin{proof}
  Let $M_p$ be the quotient of $M_{\widehat{\sigma}}$ given by dividing away the submodule generated by all $\delta_\tau$ for cells $\tau$ \emph{not} intersecting $p$. In geometric terms, $M_p$ is isomorphic to the module associated to the cone $C_p$ spanned by the faces intersecting $p$, and $\overline{\delta_\sigma}=\delta_{C_p}$. The map $M_{\widehat{\sigma}}\to \oplus_{p} M_p$ given by $\delta_\sigma\mapsto \sum_p \overline{\delta_\sigma}$ is an injection, because no cell except $\sigma$ is common to all the $C_p$; hence we have equality between the annihilator ideals $Ann_{D_X}(\delta_\sigma)= Ann_{D_X}(\sum_p \overline{\delta_\sigma})$. The latter is equal to the intersection $\cap_p Ann_{D_X}(\delta_{C_p})$.
\end{proof}

This reduces the problem to computing the annihilator ideals of the cones on the vertices of $\sigma$, which by translation is equivalent to cones at the origin. Before we give the general method, we can observe that in the special case of a simple cone we have the following nice geometric description:

\begin{Prop}\label{Prop:annihilator-simple-cone}
Let $C$ be the positive orthant in $\RR^n$.
The annihilator ideal of $\delta_C$ is the ideal $\sum_i\langle x_i\dd_i\rangle$.
\end{Prop}

\begin{proof}
(i): It is clear that $\sum_i\langle x_i\dd_i\rangle\sse Ann_{D_X}(\delta_C)$, and any $P\in Ann(\delta_C)$ can be written as $P=\sum_{ I\sse \{1,\ldots,n \} } c_I x_{I^c}^{\alpha_I}\dd_{I}^{\beta_I}$ (modulo $\sum_i\langle x_i\dd_i\rangle$), here $I^c$ denotes the complement of $I$, and $x_J^{\alpha_J}:=\prod_{j\in J}x_j^{\alpha_j}$ etc. Observe that the standard relations imply $\supp(x_{I^c}^{\alpha_I}\dd_I^{\beta_I}\cdot\delta_C)=\supp(\dd_I^{\beta_I}\cdot\delta_C)=C\cap\{x_I=0\}$, so for $P\cdot\delta_C$ to be zero, every $c_I$ must be zero, and $P$ must belong to $\sum_i\langle x_i\dd_i\rangle$.
\end{proof}

\begin{Ex}
  The standard 2-simplex in $\RR^2$ has vertices at $(0,0)$, $(1,0)$ and $(0,1)$; in coordinates $(x,y)$ the annihilator ideals of the respective cones are $\langle x \dd_x,y\dd_y\rangle$, $\langle (x+y-1)\dd_x,y(\dd_x-\dd_y)\rangle$ and $\langle (x+y-1)\dd_y,x(\dd_x-\dd_y)\rangle$. Using e.g. the \verb|Dmodules| package of the \verb|Macaulay2| computer algebra suite, we compute that the annihilator ideal is equal to $\langle x(x+y-1)\dd_x,y(x+y-1)\dd_y\rangle$.
\end{Ex}

For non-simple cones there is no neat geometric argument, but there is a general algebraic method that combines the standard relations \ref{Prop:relations} with the algebraic Fourier transform. Recall that the algebraic Fourier transform is the automorphism of $D_X$ given by $x_i\mapsto \dd_i,\dd_i\mapsto -x_i$. Twisting $M_\sigma$ with this automorphism gives the module generated by the Laplace transform $L\delta_\sigma$ of $\delta_\sigma$, and in the case where $\sigma$ is a cone at the origin, the Laplace transform is a rational function. Algorithms exist for computing the annihilator ideal of a rational function (see \cite{MR1808827} and \cite{MR1769663}), so by computing the annihilator ideal of $L\delta_\sigma$ and taking its Fourier transform, we recover the annihilator ideal of $\delta_\sigma$. Expressing $L\delta_\sigma$ as a rational function can be done by a variable elimination on the standard relations, as described below.

\begin{Alg}\label{Alg:annihilator-ideal}

Computes the annihilator ideal of a cone $\sigma$ at the origin.

\emph{Input}: the Fourier transforms of the standard relations in $M_\sigma$ of type $(i)$ (these are equations of the form $v\mu_\sigma = -\sum \langle v|n_i\rangle \mu_{\sigma_i}$, see \ref{Prop:relations}, where we denote by $\mu_\tau$ the transform of $\delta_\tau$).

\emph{Output}: the annihilator ideal of $\delta_\sigma$.
\begin{enumerate}
\item Considering all the $\mu_{\sigma_i}$'s as formal variables, eliminate from the transformed equations all the $\mu_{\sigma_i}$ except for $\mu_\sigma$ itself; this expresses $\mu_\sigma$ as a rational function in the original variables $x_i$.
\item Compute the annihilator ideal of this rational function.
\item Take the Fourier transform of this ideal.
\end{enumerate}
\end{Alg}

\begin{Ex}
  Let $\sigma$ be the cone in $\RR^3$ with rays $(1,0,0)$, $(0,1,0)$, $(0,0,1)$ and $(-1,-1,-1)$ (connected in that order). The Laplace transform of $\delta_\sigma$ is $L\delta_\sigma=\frac{x+z}{xyz(x+y+z)}$. The annihilator ideal of $L\delta_\sigma$ is generated by the elements $x\dd_x+y\dd_y+z\dd_z+3$, $z\dd_y\dd_z-z\dd_z^2+\dd_y-2\dd_z$, $yz(\dd_y^2-\dd_z^2)+2z\dd_y-2y\dd_z-2z\dd_z-2$, $yz\dd_x(\dd_y-\dd_z)+z(\dd_x-\dd_z)-y\dd_y-2$ and $yz(x+z)(\dd_y-\dd_z)-xy+xz+z^2$. Its Fourier transform is the annihilator ideal of $\delta_\sigma$, and is generated by $x\dd_x+y\dd_y+z\dd_z$, $(y-z)z\dd_z$, $(y^2-z^2)\dd_y\dd_z+2z\dd_z$, $(y-z)x\dd_y\dd_z+(y-x)\dd_y+z\dd_z$ and $(z-y)(\dd_x+\dd_z)\dd_y\dd_z+2\dd_y\dd_z$.
\end{Ex}


\section{Direct images and $B$-splines}\label{sec:dir-im}

We consider the projection map $\pi:\RR^m=\RR^s\times\RR^{m-s}\to \RR^s$ given by projecting on the first $s$ coordinates, and its complexification $\pi:\CC^m=\CC^s\times\CC^{m-s}\to \CC^s$. For convenience we denote as before $X=\CC^m$, and also $Y=\CC^s,Z=\CC^{m-s}$, and use coordinates $(y,z)$ on $X$. The direct image functor from $D_X$-modules to $D_Y$-modules is given by
$$\pi_+M:=DR_{X/Y}(M),$$
see \cite{Malgrange},\cite[V,VI]{Borel} for details. In considering the direct image of $M_K$, we will focus in particular on the zeroth-level part 
\begin{equation}\label{eqn:proj-quot}\pi_+^0M_K\simeq M_K/\sum_j\dd_{z_j}M_K.\end{equation}

Some useful notation:

\begin{Defn}\label{Defn:fiber-dim}For a cell $\sigma\sse K$, let $v(\sigma):=\dim(\sigma)-\dim(\pi(\sigma))$ (which is the same as the dimension of a generic fiber $\pi^{-1}(x)\cap K$ for a point $x\in int(\pi(\sigma))$). We call $v(\sigma)$ the \emph{fiber dimension} of $\sigma$. In particular, if  $\dim(\sigma)=\dim(X)=m$, then $v(\sigma)=m-s$.

We also extend this notation to the whole complex, and let $v(K):=\dim(K)-\dim(\pi(K))$. Note that a complex $K$ can contain cells $\sigma$ with $v(\sigma) > v(K)$. 
\end{Defn}

\begin{Prop}[Standard relations for $\pi_+^0M_K$]\label{Prop:pi_+^0-rels}
  Let $\sigma$ be a cell in $K$ of top dimension (i.e. $\dim(\sigma)=\dim(X)=m$), with facets $\sigma_i$ with outward unit normals $n_i$, and let $\pi:X=Y\times Z\to Y$ be the projection on the first $s$ coordinates. Denote the class of $\delta_\sigma$ in the direct image $\pi_+^0M_K$ by $\overline{\delta_\sigma}$.
Then the following relations hold:
\begin{enumerate}
\item[(i)]$\dd_{\pi(z)}\overline{\delta_\sigma} = -\sum_i\langle \pi(z) | n_i\rangle \overline{\delta_{\sigma_i}}$, for any point $z$ in $H_\sigma$ (where we let $\dd_{\pi(z)}:=\sum_i\langle e_i|\pi(z)\rangle \dd_i$),
\item[(ii)]$\sum_i\langle v|n_i\rangle\overline{\delta_{\sigma_i}}=0$, for any $v\in \ker(\pi)$, and
\item[(iii)]$v(\sigma)\overline{\delta_\sigma} = \sum_i(d_i-\sum_{j\le s}\langle e_j|n_i\rangle x_j)\cdot\overline{\delta_{\sigma_i}}$, where $\sum_j\langle e_j|n_i\rangle x_j-d_i=0$ is the defining equation of $H_{\sigma_i}$.
\item[(iv)]$p(x)\cdot\overline{\delta_\sigma}=0$ for any $p(x)\in I(H_{\pi(\sigma)})$.
\end{enumerate}
\end{Prop}

\begin{proof}For $j\le s$, the action of $\dd_j$ is unchanged in the quotient (\ref{eqn:proj-quot}), which implies $(i)$. For $j>s$, $\dd_j\delta_\sigma$ is zero in the quotient (\ref{eqn:proj-quot}): $0=\dd_j\delta_\sigma=\sum_i\langle e_j|n_i\rangle\overline{\delta_{\sigma_i}}$ and since $\ker(\pi)=\langle e_j|j>s\rangle$, we get $(ii)$.

The affine spans $H_{\sigma_i}$ of the boundary cells $\sigma_i$ are defined by equations $\langle x|n_i\rangle=d_i$ for some constants $d_i$. Now $\sum_{j>s} \dd_{j}x_{j}\overline{\delta_\sigma}=0$, because $\sum \dd_jx_j$ is in the ideal $\sum_{j>s} \dd_jD_X$. We then get $0=\sum_{j>s}\dd_jx_j\overline{\delta_\sigma}=\sum_{j>s}(1+x_j\dd_j)\overline{\delta_\sigma}$, or (using $v(\sigma)=m-s$) $$\left((m-s)+\sum_{j>s}x_j\dd_j\right)\overline{\delta_\sigma}=\left(v(\sigma)+\sum_{j>s}x_j\dd_j\right)\overline{\delta_\sigma}=0.$$
Let us expand this:
\begin{eqnarray*}v(\sigma)\overline{\delta_\sigma} &=&-\sum_{j>s}x_j\dd_j\overline{\delta_\sigma}\\
&=&\sum_{j>s}x_j\sum_i\langle e_j|n_i\rangle \overline{\delta_{\sigma_i}}\\
 &=& \sum_i(\sum_{j>s}\langle e_j|n_i\rangle x_j)\overline{\delta_{\sigma_i}}\\
&=&\sum_i(d_i-\langle e_1|n_i\rangle x_1-\cdots -\langle e_{s}|n_i\rangle x_{s})\overline{\delta_{\sigma_i}}\\
&=&\sum_i(d_i-\sum_{j\le s}\langle e_j|n_i\rangle x_j)\overline{\delta_{\sigma_i}},
\end{eqnarray*}
where the second-to-last equality uses the standard relation $(\sum_j\langle e_j|n_i\rangle x_j-d_i)\delta_{\sigma_i}=0$; and we have $(iii)$. The claim $(iv)$ is of course obvious, it follows by definition that $\supp(\overline{\delta_\sigma})= \pi(\supp(\delta_\sigma))=\pi(H_\sigma)=H_{\pi(\sigma)}$. 
\end{proof}

The corresponding result for general cells not of top dimension follows from an application of Kashiwara's Theorem.

\begin{Defn}\label{Defn:pi-skel}
  Let $F_{\le i}^\pi:=D_X\cdot\{\delta_\sigma|\sigma\sse K,\dim(\pi(\sigma))\le i\}$. These submodules form a filtration on $M_K$, which we call the \emph{$\pi$-skeleton filtration}. We denote the filtration quotients by $Q_i^\pi$.
  Let $F_{\le i}^{\pi'}:=D_Y\cdot\{\overline{\delta_\sigma}|\sigma\sse K,\dim(\pi(\sigma))\le i\}$. These are submodules of $\pi_+^0M_K$, and form a filtration which we call the \emph{skeleton filtration on $\pi_+^0M_K$}. We denote the filtration quotients by $Q_i^{\pi'}$.
\end{Defn}

\begin{Defn}
  Let $K_i,K_{\le i}$ denote the subcomplexes of $K$ given by respectively $K_i:=\bigcup_{\dim(\pi(\sigma))= i}\sigma$ and $K_{\le i}:=\bigcup_{\dim(\pi(\sigma))\le i}\sigma$. Note that $K_{\le i}$ is closed in $K$, and so clearly $F_{\le i}^\pi\simeq M_{K_\le i}$.
\end{Defn}

\begin{Prop}\label{Prop:pi-skel-dir-im}
  $F_{\le i}^{\pi'}\simeq \pi_+^0F_{\le i}^\pi$, and $Q_{i}^{\pi'}\simeq \pi_+^0Q_{i}^\pi$.
\end{Prop}

\begin{proof}
The second claim follows from the first. There is a surjective map $\theta:\pi_+^0F_{\le k}^\pi\to F_{\le k}^{\pi'}$ given by $\overline{\delta_\sigma}\mapsto \overline{\delta_\sigma}$ sending the class of $\delta_\sigma$ in $\pi_+^0F_{\le k}^\pi$ to the class of $\delta_\sigma$ in $\pi_+^0M_K$. If we let $\iota:F_{\le k}^{\pi'}\inj \pi_+^0M_K$ and $i:F_{\le k}^\pi\inj M_K$ be the inclusions of submodules, and $\widehat{i}:\pi_+^0F_{\le k}^\pi\to \pi_+^0M_K$ the map induced by $i$ in the direct image, we have $\widehat{i}=\iota\circ\theta$, so if we can show that $\widehat{i}$ is injective, $\theta$ is an isomorphism.

Consider now the diagram 
$$\xymatrix{F_{\le k}^\pi\ar@{^{(}->}[d]_i\ar@{>>}[r]^{q_k} & \pi_+^0F_{\le k}^\pi\ar[d]_{\widehat{i}} \\M_K\ar@{>>}[r]^q\ar@{>>}[d] & \pi_+^0M_K\ar@{>>}[d] \\ \cok(i)\ar@{>>}[r]^{\overline{q}} & \cok(\widehat{i}) }$$
where $q: M_K\surj M_K/\sum \dd_zM_K\simeq \pi_+^0M_K$ (and similar for $q_k$) is the quotient map from (\ref{eqn:proj-quot}), and $\overline{q}:\cok(i)\to\cok(\widehat{i})$ is the map induced by $q$. Considering this diagram as a double complex, the associated spectral sequence gives us that $\ker(\widehat{i})=0$ and we are done.\end{proof}

\begin{Thm}\label{Thm:rel-de-Rham}The cohomology modules $h^i(DR_{X/Y}Q_k^\pi)$ of $DR_{X/Y}Q_k^\pi$ are semisimple $D_Y$-modules, with summands isomorphic to $D_Y\cdot\overline{\delta_{\pi(\sigma)}}$ for $\sigma\in K_k$. The number of such summands for $h^i(DR_{X/Y}Q_k^\pi)$ is equal to $\dim H_{i+k}^{BM}(K_k,\CC)$.
\end{Thm}

\begin{proof}We recall our convention that $X=Y\times Z$, with $\pi$ the projection on $Y$. We will (begin to) compute the relative de Rham complex by means of a skeleton filtration (as in \ref{Defn:skel-filt} with the obvious alterations) on $Q_k^\pi$. We can express each skeleton filtration quotient summand $D_X\cdot \overline{\delta_\sigma}$ as a module $\CC[y_1,\ldots,\dd_{y_s},z_1,\ldots,\dd_{z_{m-s}}]$ by choosing suitable coordinates, in the following manner. We choose the $y_i$ such that
$$D_Y\cdot \overline{\delta_{\pi(\sigma)}}\simeq \CC[y_1,\ldots,y_{\dim(\pi(\sigma))},\dd_{y_{\dim(\pi(\sigma))+1}},\ldots,\dd_{y_s}]$$ 
in the same way as in \ref{Rem:tool}. Similarly, we choose the $z_j$ such that for a generic fiber $F:=\pi^{-1}(p)\cap H_\sigma$ (where $p\in int(\pi(\sigma))$ is some point), we have, also as in \ref{Rem:tool}, that 
$$D_Z\cdot\overline{\delta_F}\simeq\CC[z_1,\ldots,z_{v(\sigma)},\dd_{z_{v(\sigma)-1}},\ldots,\dd_{z_{m-s}}].$$ 
This isomorphism of course depends on which point $p$ we choose, but the coordinates do not. In particular we have $\dd_{z_j}\overline{\delta_\sigma}=0$ for $j\le v(\sigma)$. 

The relative de Rham complex $DR_{X/Y}(D_X\cdot\overline{\delta_\sigma})$ is now of the form 
$$\Omega_{X/Y}^\bullet\ten \CC[y_1,\ldots,y_{\dim(\pi(\sigma))},\dd_{y_{\dim(\pi(\sigma))+1}},\ldots,\dd_{y_s},z_1,\ldots,z_{v(\sigma)},\dd_{z_{v(\sigma)-1}},\ldots,\dd_{z_{m-s}}],$$ 
and since the differential $d_Z=\sum_jdz_j\ten \dd_{z_j}$ commutes with the $Y$ variables, this becomes 
$$\CC[y_1,\ldots,\dd_{y_s}]\ten_{\CC}\left(\Omega_{Z}^\bullet\ten_{\CC[Z]} \CC[z_1,\ldots,\dd_{z_{m-s}}]\right).$$ 
As in \ref{Thm:deRham} we have 
$$\Omega_Z^\bullet\ten \CC[z_1,\ldots,z_{v(\sigma)},\dd_{z_{v(\sigma)-1}},\ldots,\dd_{z_{m-s}}]\simeq C_\bullet^{BM}(\pi^{-1}(p)\cap \mathring{\sigma}).$$ 
The cohomology is now computed via the spectral sequence associated to the skeleton filtration on $Q_k^\pi$, which begins with $E_0^{pq}=\Omega_{X/Y}^{m-s+p+q}\ten Q_{-p}$, and collapses on the $E_1$ page to a single row 
$$\bigoplus_{v(\sigma)=m-s-k} D_Y\cdot\overline{\delta_{\pi(\sigma)}}\to \cdots \to \bigoplus_{v(\sigma)=0} D_Y\cdot\omega_\sigma\ten\overline{\delta_{\pi(\sigma)}}$$ (we let $\omega_\sigma=dz_{1}\land\cdots\land dz_{v(\sigma)}$, in the coordinates suiting each $\sigma$ as above).

We can now show that the cohomology modules must be direct sums of simple modules: as the differential $d_Z$ commutes with $D_Y$, its action is determined by the action on the generators $\omega_{\pi(\sigma)}\overline{\delta_{\pi(\sigma)}}$, and so taking cohomology only involves identification of generators. This implies that the cohomology modules are of the form $\sum D_Y\cdot\omega_{\pi(\sigma)}\ten \overline{\delta_{\pi(\sigma)}}$, and one gets a (non-canonical) direct sum decomposition by choosing some generating set. We recall from \ref{Cor:simple} that each summand $D_Y\cdot \omega_{\pi(\sigma)}\ten\overline{\delta_{\pi(\sigma)}}$ is simple.

We want to relate this to the homology of $K_k$. We recall from \ref{Thm:deRham} that the de Rham differential $d_X$ corresponds to the topological boundary map, because for a generator $\omega_\sigma\ten\delta_\sigma$ we had $$d(\omega_\sigma\ten\delta_\sigma) = -\sum_{\sigma_i \sse \dd\sigma}\omega_{\sigma_i}\ten\delta_{\sigma_i}.$$ Now, in the relative de Rham complex we have the relative differential $d_Z$ acting on generators $\omega_\sigma\ten\delta_\sigma$, and this also behaves like the topological boundary map, the same computation as in \ref{Thm:deRham} works, and we get 
$$d_Z(\omega_\sigma\ten\delta_\sigma)=-\sum_{\sigma_i\sse \dd\sigma}\omega_{\sigma_i}\ten \delta_{\sigma_i}$$
We see that the correspondence of the de Rham differential to the topological boundary map holds, except for one subtlety: those cells $\sigma_i\sse\dd\sigma$ such that $H_{\sigma_i}$ is contained in a translate of $\ker(\pi)$ do not appear in the final sum. These are precisely those cells in the boundary of $\sigma$ that have image of dimension strictly lower than $\dim(\pi(\sigma))$. Thus, if we restrict our attention to the subcomplex $K_k$, where these cells are removed, the correspondence to the topological boundary map remains. Just as we had $d[\sigma]=d(\omega_\sigma\ten\delta_\sigma)=-\sum \omega_{\sigma_i}\ten\delta_{\sigma_i}=[\dd\sigma]$, we have now $d_Z(\omega_\sigma\ten\delta_\sigma)=-\sum \omega_{\sigma_i}\ten\delta_{\sigma_i}$, the only difference is instead of constant coefficients we now have $D_Y$-coefficients.

We observe that the cells in $K_k$ all have $\dim(\sigma)=k+v(\sigma)$, and accordingly the generators of $h^i(DR_{X/Y}Q_k^\pi)$ correspond to cells with $\dim(\sigma)=k+i$. This gives us that the number of summands in $h^i(DR_{X/Y}Q_k^\pi)$ is equal to the dimension of the homology group $H_{k+i}^{BM}(K_k,\CC)$, and one can choose as generators any set of $\overline{dz_{J_{\sigma}}\ten\delta_\sigma}$'s such that the associated homology classes $[\sigma]$ generate $H_{k+i}^{BM}(K_k,\CC)$.
\end{proof}

\begin{Rem}
  Using \ref{Thm:rel-de-Rham} we can compute the skeleton filtration quotients of each level of the direct image $\pi_+M_K$, by running the appropriate spectral sequence. To recover $\pi_+M_K$ from the filtration quotients, it is enough to find the extension with the correct de Rham cohomology, as each isomorphism class of extensions has different cohomology.
\end{Rem}

In the same way as \ref{Prop:presentation} we can show the following:

\begin{Prop}\label{Prop:presentation-dir-im}
  There is a canonical presentation $$(D_Y)^r\to (D_Y)^c\surj \pi_+^0M_K$$ where $c$ is the number of cells in $K$, and $r$ is equal to $(\dim(Y)+1)\cdot c+\sum_{\sigma\sse K}(v(\sigma)-\delta_{0,v(\sigma)})$ (here, $\delta_{0,v(\sigma)}$ is the Kronecker delta function). 
\end{Prop}


\subsection{Distributional direct images and $B$-splines}\label{subsec:spline}

\begin{Defn}
  The \emph{distributional direct image} of $\delta_\sigma$ is defined by
\[\pi_*\delta_\sigma := [\phi\mapsto \int_\sigma \phi\circ \pi\; dx]\]
for a test function $\phi$ on $\RR^s$. This is the distribution form of the well-known \emph{multivariate $B$-spline}
$$\sigma_\pi(x)=\frac{1}{\sqrt{\det(\pi\cdot\pi^t)}}\vol(\pi^{-1}(x)\cap \sigma),$$ 
(where we by abuse of notation write $\pi$ for the matrix associated to $\pi$). Using this we can also express $\pi_*\delta_\sigma$ as the distribution
$$\phi\mapsto\int_{\RR^s}\phi(x)\sigma_\pi(x)dx$$
(see e.g. \cite[chapter 7]{DeConcini-Procesi} for details).
\end{Defn}

De Concini and Procesi in \cite{DeConcini-Procesi} investigate some of the properties of the module $D_Y\cdot\pi_*\delta_K$, when $\pi$ is a projection, in the special cases when $K$ is a box or a cone. In light of what we have done so far, we might say that for general $K$, the module generated by all the $\pi_*\delta_{\sigma}$ is the more natural object, so let us investigate it closer. 

\begin{Defn}We let $S_K:=W\cdot\{\pi_*\delta_\sigma|\sigma\subset K\}$.
\end{Defn}

There are similar standard relations as for $\pi_+^0M_K$:

\begin{Thm}[De Boor - Höllig, \cite{DB-H}]\label{Thm:DB-H}  Let $\sigma$ be a polyhedral body in $\RR^m$, with facets $\sigma_i$, and corresponding outward unit normals $n_i$, and let $\pi$ be the projection on the first $s$ coordinates. Assume also that the fibers $\pi^{-1}(x)\cap K$ are compact. Then the following hold:

\begin{enumerate}
\item[(i)]$\dd_{\pi(z)}\pi_*\delta_\sigma = -\sum_i\langle \pi(z) | n_i\rangle \pi_*\delta_{\sigma_i}$, for any $z\in \RR^m$,
\item[(ii)]$\sum_i\langle v|n_i\rangle\pi_*\delta_{\sigma_i}=0$, for $v\in \RR^m$ orthogonal to $\RR^s$, and
\item[(iii)]$v(\sigma)\pi_*\delta_\sigma = \sum_i\langle k_i-x|n_i\rangle \pi_*\delta_{\sigma_i}$, where $k_i$ is an arbitrary point of $\sigma_i$ and $x\in \RR^s$.
\end{enumerate}
\end{Thm}

\begin{Rem}
  As with the previous case, suitable restrictions of $\pi$ to $H_\sigma$ with appropriate coordinate changes give the corresponding results for general $\sigma$.
\end{Rem}

We observe that the standard relations for $\pi_*\delta_\sigma$ (\ref{Thm:DB-H}) are essentially identical to the standard relations for $\pi_+^0M_K$ (\ref{Prop:pi_+^0-rels}), and we can make the analogous constructions of skeleton filtration and canonical presentation, and achieve the analogous results (we omit tedious repetition of the arguments). 

There is only one difference between the modules $\pi_+^0M_K$ and $S_K$ defined by these standard relations: \ref{Thm:DB-H}$(ii)$ and \ref{Prop:pi_+^0-rels}$(ii)$ both essentially say that $\dd_v\delta_\sigma=0$ for $v\in\ker(\pi)$, which means a certain linear combination of the boundary cells $\delta_{\sigma_i}$ is zero. The important observation is that \ref{Thm:DB-H}$(ii)$ applies even if the $\sigma$ in question is \emph{not} in $K$, while \ref{Prop:pi_+^0-rels}$(ii)$ does not. The reason is obvious: the $\pi_*\delta_{\sigma}$, being concrete distributions, do not care what module they sit in, while the abstract generators $\overline{\delta_\sigma}$ are not so lucky. We thus get extra relations in $\pi_+^0M_K$ whenever 
there are ``missing'' cells.

To formalise this we recall a definition from general topology (see also \cite{Cohen-SimpHtpy} for a further introduction):

\begin{Defn}\label{Defn:elementarily-equivalent}
  We say that a cell $\sigma$ in $K$ has a \emph{free facet} $\tau$ if $\tau$ is a facet of $\sigma$, and is not a facet of any other cell in $K$; we say that $(\sigma,\tau)$ is a \emph{free pair}. If we remove $\sigma$ and $\tau$ from $K$, we obtain another cell complex $L$ which we call an \emph{elementary collapse} of $K$, and $K$ is an \emph{elementary expansion} of $L$. A complex $L$ obtained from $K$ by a sequence of elementary collapses is called a \emph{collapse} of $K$, and we say that $K$ is an \emph{expansion} of $L$. Two complexes related by a sequence of collapses and expansions are said to be \emph{elementarily equivalent}.
\end{Defn}

Let us modify this slightly to suit our purposes:

\begin{Defn}\label{Defn:v-elementarily-equivalent}
  If $(\sigma,\tau)$ is a free pair of $K$ and $v(\sigma)=v(\tau)+1$, we say that $(\sigma,\tau)$ is a \emph{$v(\sigma)$-free pair} of $K$ (with respect to $\pi$). The concepts of $v(\sigma)$-(elementary) collapse and $v(\sigma)$-elementary equivalence are defined analogously (with all involved elementary collapses and extensions being the removal or addition of a $v(\sigma)$-free pair). 
\end{Defn}

\begin{Thm}\label{Thm:mainThm}
  There is a canonical surjective map  $\pi_+^0M_K\to S_K$ given by $\overline{\delta_\sigma}\mapsto \pi_*\delta_\sigma$. If $K$ is 1-elementarily equivalent to a complex $K'$ with connected fibers $\pi^{-1}(x)\cap K'$, then the canonical map is an isomorphism.
\end{Thm}

\begin{proof}Surjectivity follows directly from the above observations about the standard relations. It remains to show two things: that 1-elementary collapses do not change the isomorphism class of $\pi_+^0M_K$; and that if $K$ has connected fibers under $\pi$, then $\pi_+^0M_K$ is isomorphic to $S_K$. 

The standard relation \ref{Prop:pi_+^0-rels}$(iii)$ expresses each $\overline{\delta_\sigma}$ with $v(\sigma)> 0$ in terms of those of its facets $\overline{\delta_{\sigma_i}}$ with $v(\sigma_i)<v(\sigma)$ (the coefficients of the remaining facets are zero). By repeated application, this implies that $\pi_+^0M_K$ is generated by those $\overline{\delta_\sigma}$ for which $v(\sigma)=0$, with relations among them determined by cells with $v(\sigma)=1$ (given by \ref{Prop:pi_+^0-rels}$(ii)$) and cells with $v(\sigma)=0$ (given by \ref{Prop:pi_+^0-rels}$(i)$). (The analogous statement for $S_K$ follows in the same manner.) It follows that if we add or remove cells from $K$ to produce another complex $K'$, we get isomorphic direct image modules if the addition or removal of cells preserves these relations. The claim is thus that 1-elementary collapses and expansions preserves the standard relations. 

It suffices to check this for a complex $\widehat{\sigma}$ with $v(\sigma)=1$. The standard relation \ref{Prop:pi_+^0-rels}$(ii)$ essentially says that $\dd_v\delta_\sigma=0$ for any $v\in\ker(\pi)$ parallel to $H_\sigma$, which means a certain linear combination of those boundary cells $\delta_{\tau}$ with $v(\tau)=0$ is zero, and we can thus write any one of them as a sum of the others, which then generate $\pi_+^0M_{\widehat{\sigma}}$. So, in $\pi_+^0M_{\widehat{\sigma}}$, any single one of the generators $\overline{\delta_\tau}$ (with $\tau$ a free facet of $\sigma$) is redundant, and it follows that $\pi_+^0M_{\widehat{\sigma}}\simeq \pi_+^0M_{\widehat{\sigma}\setminus (\sigma,\tau)}$, and further that 1-elementary equivalent cell complexes give isomorphic direct image modules.

For the second claim, if $K$ has connected fibers, all the cells with $v(\tau)=0$ are connected by cells with $v(\sigma)=1$, and so adding any more cells with $v(\sigma)=1$ can not introduce any extra relations between the generators; the `extra' relations in $S_K$ are already there.
\end{proof}

Even when \ref{Thm:mainThm} fails, we can at least express $\pi_+^0M_K$ as an extension, using the geometry of $K$ and $\pi$ to recover the kernel of the map $\pi_+^0M_K\surj S_K$. We illustrate by a simple example:

\begin{Ex}
Let $\pi:\RR^3\to\RR^2$ be the projection $(x,y,z)\mapsto (x,y)$, and let $K$ be the unit box $[0,1]^3$ with the interior and any two `vertical' facets removed. It is easy to see that $K$ is not 1-elementarily equivalent to a complex with connected fibers, as neither of the missing facets form a free pair with the interior of the box, because of the remaining missing facet. 

Letting $top$ and $bottom$ denote the top and bottom facets (in the $z$ direction), we see that the kernel of the map $\pi_+^0M_K\surj S_K$ is generated by $\overline{\delta_{top}}-\overline{\delta_{bottom}}$ (considered as a submodule of $\pi_+^0M_K$). This submodule is isomorphic to the module generated by $\delta_{\pi(K)}$ (and in this case actually isomorphic to $S_K$, though this is not the general case), in other words we have the exact sequence $0\to M_{\pi(K)}\to \pi_+^0M_K\to S_K\to 0$.
\end{Ex}

\section*{Acknowledgements}

I am grateful to my advisor Rikard Bøgvad, for all the usual reasons; I would also like to thank Rolf Källström and Jan-Erik Björk for helpful discussions. 


\begin{thebibliography}{10}

\bibitem{Borel}
A.~Borel, P.-P. Grivel, B.~Kaup, A.~Haefliger, B.~Malgrange, and F.~Ehlers,
  \emph{Algebraic {$D$}-modules}, Perspectives in Mathematics, vol.~2, Academic
  Press Inc., Boston, MA, 1987.

\bibitem{Cohen-SimpHtpy}
Marshall~M. Cohen, \emph{A course in simple-homotopy theory}, Springer-Verlag,
  New York, 1973, Graduate Texts in Mathematics, Vol. 10.

\bibitem{Coutinho}
S.~C. Coutinho, \emph{A primer of algebraic {$D$}-modules}, London Mathematical
  Society Student Texts, vol.~33, Cambridge University Press, Cambridge, 1995.

\bibitem{DB-H}
Carl de~Boor and Klaus H{\"o}llig, \emph{Recurrence relations for multivariate
  {$B$}-splines}, Proc. Amer. Math. Soc. \textbf{85} (1982), no.~3, 397--400.

\bibitem{DeConcini-Procesi}
Corrado De~Concini and Claudio Procesi, \emph{Topics in hyperplane
  arrangements, polytopes and box-splines}, Universitext, Springer, New York,
  2011.

\bibitem{Kashiwara-thesis}
Masaki Kashiwara, \emph{Algebraic study of systems of partial differential
  equations}, M\'em. Soc. Math. France (N.S.) (1995), no.~63, xiv+72.

\bibitem{Malgrange}
Bernard Malgrange, \emph{De {R}ham complex and direct images of
  {$\mathscr{D}$}-modules}, \'{E}l\'ements de la th\'eorie des syst\`emes
  diff\'erentiels. {I}mages directes et constructibilit\'e ({N}ice, 1990),
  Travaux en Cours, vol.~46, Hermann, Paris, 1993, pp.~1--13.

\bibitem{MR1808827}
Toshinori Oaku and Nobuki Takayama, \emph{Algorithms for
  {$D$}-modules---restriction, tensor product, localization, and local
  cohomology groups}, J. Pure Appl. Algebra \textbf{156} (2001), no.~2-3,
  267--308.

\bibitem{MR1769663}
Toshinori Oaku, Nobuki Takayama, and Uli Walther, \emph{A localization
  algorithm for {$D$}-modules}, J. Symbolic Comput. \textbf{29} (2000),
  no.~4-5, 721--728, Symbolic computation in algebra, analysis, and geometry
  (Berkeley, CA, 1998).

\bibitem{Pham}
Fr{\'e}d{\'e}ric Pham, \emph{Singularit\'es des syst\`emes diff\'erentiels de
  {G}auss-{M}anin}, Progress in Mathematics, vol.~2, Birkh\"auser Boston,
  Mass., 1979, With contributions by Lo Kam Chan, Philippe Maisonobe and
  Jean-{\'E}tienne Rombaldi.

\end{thebibliography}

\end{document}